\documentclass[a4paper,12pt,reqno]{amsart}
\usepackage[utf8]{inputenc}
\usepackage{amssymb,amsthm}
\usepackage{enumerate}
\usepackage{graphicx}
\usepackage{color}

\newtheorem{theorem}{Theorem}[section]

\newtheorem{cor}[theorem]{Corollary}

\newtheorem{remark}[theorem]{Remark}
\newtheorem{definition}[theorem]{Definition}

\newcommand{\R}{\mathbb{R}}

\newcommand{\N}{\mathbb{N}}
\newcommand{\Z}{\mathbb{Z}}

\title[Bendixson--Dulac theorem]
{Effectiveness\\ of the Bendixson--Dulac theorem}

\author[A. Gasull]{Armengol Gasull}
\address{Departament de Matem\`{a}tiques, Edifici Cc, Universitat Aut\`{o}noma de Barcelona, 08193 Cerdanyola del Vall\`{e}s (Barcelona), Spain.}
\address{Centre de Recerca Matem\`{a}tica, Edifici Cc, Campus de Bellaterra, 08193 Cerdanyola del Vall\`{e}s (Barcelona), Spain.}
\email{gasull@mat.uab.cat}

\author[H. Giacomini]{Hector Giacomini}
\address{ Institut Denis Poisson. Universit\'{e} de Tours,  C.N.R.S. UMR 7013. 37200 Tours, France.} \email{Hector.Giacomini@lmpt.univ-tours.fr}

\date{}

\begin{document}

\begin{abstract} We illustrate with several new applications the power and elegance of the
Bendixson--Dulac theorem to obtain upper bounds of the number of
limit cycles for several families of planar vector fields. In some
cases we propose to use a function related with the curvature of the
orbits of the vector field as a Dulac function. We get some general
results for Li\'{e}nard type equations and for rigid planar systems. We
also present a remarkable phenomenon: for each integer $m\ge2,$ we
provide a simple $1$-parametric differential system for which we
prove that it has limit cycles only for the values of the parameter
in a subset of an interval of length smaller that
$3\sqrt{2}(3/m)^{m/2}$ that decreases exponentially when $m$ grows.
One of the strengths of the results presented in this
work is that although they are obtained with simple calculations,
that can be easily checked by hand, they improve and extend previous
studies. Another one is that, for certain systems, it is possible to reduce the question of the number of limit cycles to the study of the shape of a planar curve and the sign of an
associated function in one or two variables.
\end{abstract}

\maketitle

\section{Introduction}

Despite all the efforts dedicated to solve the second part of the
Hilbert's 16th problem, it is yet a very difficult task to obtain
criteria that give explicit upper bounds for many concrete families
of planar  smooth vector fields. Although there is no any universal
approach, the aim of this paper is to present several families of planar
systems for which the Bendixson--Dulac theorem allows to get, in a
fast and elegant way, an upper bound of their number
of limit cycles. We will avoid results based on  cumbersome
computations.

The families that we will consider include extensions of Li\'{e}nard
systems and rigid systems. As we will see, we obtain new results and
we also present simple proofs of some recent results. They give
explicit upper bounds for several families of planar vector fields. These
upper bounds are also sharpened when we deal with more particular
systems, obtaining results of at most two, one, or none limit cycles.

Our main results for Li\'{e}nard type  systems are contained in
Section~\ref{se:lie}. They are  given in Theorem~\ref{th:wil}, that
deals with a version of Wilson Li\'{e}nard systems which always have an algebraic limit cycle,
 in Theorem~\ref{th:nou} that
studies a family recently introduced in \cite{VilZan2020}, in
Theorem~\ref{th:mas} that extends the classical theorem of Massera,
and in Theorem~\ref{th:vil}. In fact, this last result includes the
remarkable phenomenon highlighted in the abstract: the family
\begin{equation*}
\begin{cases}
\dot x=y-\lambda |y|^{m}(x^3-x),\\
\dot y=-x,
\end{cases}
\end{equation*}
introduced in \cite{VilZan2020}, has for $m\ge2$, limit cycles only
for some values of $\lambda$ contained in the interval of length
$3\sqrt{2}( 3 /m)^{m/2},$ centered at the origin. Notice that for
$m$ big it is extremely thin. This interval decreases
exponentially with $m$. \\
Our main result for general rigid systems is given in Theorem
\ref{te:rig} of Section \ref{se:rig}. It is applied to recover, in a
simple way, known results for rigid cubic systems and to a family
containing non polynomial vector fields.

It seems to us that not all the mathematical community that works on
these topics is aware of the capability of the
Bendixson--Dulac approach. The goal of this work is double: we  try
to change this perception and we also present several new results
and easy proofs of some known results. For instance, in most
textbooks, the proof of the uniqueness and hyperbolicity of the
limit cycle for the classical van der Pol equation needs some work.
By using this approach  there are extremely simple proofs, see
Corollary~\ref{co:vdp} and Remark~\ref{re:vdp}.

The today known as Bendixson--Dulac theorem  was first formulated by
Ivar Bendixson in 1901 (\cite{Ben1901}), and later developed by
Henri Dulac in 1937 (\cite{Dul1937}). He improved Bendixson's result
by introducing a new parameterization of the time, via the today
called Dulac functions. This result appears, under different versions,
 in most differential equations textbooks.
One of the pioneers to try to go further with this approach was
Yamato (\cite{Yam1979}). Afterward, one of its main defenders was
Cherkas, who used and developed it, see for instance
\cite{Che1997,CheGri2010}. The authors of this work also often apply
and try to extend this method, see \cite{GasGia2002,GasGia2006,GasGia2010}. More
examples about its applicability can be seen in the survey
\cite{GasGia2013}.

In this paper we will use the version of the Bendixson--Dulac
theorem that we state below, after introducing some notations and
definitions. For completeness, in Section \ref{se:pre} we present a
proof based on the version of Bendixson--Dulac theorem for multiply
connected regions that is proved for instance in
\cite{Che1997,GasGia2002,Llo1979}.

Given an open connected subset $\mathcal{U}\subset\R^2,$ with
finitely many holes, we will denote by $\ell=\ell(\mathcal{U})$ this
number of holes, that is,  the number of bounded components of
$\R^2\setminus\mathcal{U}.$ Notice that if $\mathcal{U}$ is simply
connected then $\ell(\mathcal{U})=0.$ We also set
$\R^+=\{x\in\R\,:\, x>0\}.$

 For a continuous function $f:\R^n\to\R,$ not changing sign and
vanishing on a null measure set, we will denote by
$\operatorname{sign}(f)$ the sign of $f$ at any of its  point where
it is not zero. Moreover, given an equilibrium point or a periodic
orbit, when we say that its stability is given by the sign of $f$ we
mean that the object is an attractor (resp. a repeller) whenever
$\operatorname{sign}(f)<0$ (resp. $\operatorname{sign}(f)>0$).

\begin{definition}\label{def} Given a function  $V:\R^2\to\R$ of class
$\mathcal{C}^1$ we will say that it is {\rm admissible} if:
\begin{enumerate}[(i)]
\item The vector $\nabla V$ vanishes on $\{V(x,y)=0\}$ at finitely many points.
\item The set $ \{V(x,y)=0\}$ has finitely many
connected components.
\item The set $\R^2\setminus \{V(x,y)=0\}$ has $j$ connected components,
$\mathcal{U}_i, i=1,2,\ldots j,$ and for all of them
$\ell(\mathcal{U}_i)<\infty.$
\end{enumerate}
Associated to $V,$ we define the non negative integer
number
\[L(V):=\sum_{i=1}^j \ell(\mathcal{U}_i).\]
\end{definition}

\begin{theorem}[Bendixson--Dulac theorem]\label{th:bd}
Consider a $\mathcal{C}^1$ planar differential system
\begin{equation}\label{eq:bd}
\dot x= P(x,y),\quad \dot y=Q(x,y),
\end{equation}
and denote by $X=(P,Q)$ its associated vector field. Let
$V:\R^2\to\R$ be an admissible function such that there exists
$s\in\R^+$ for which the function
\begin{equation}\label{eq:ms} M_s:=\frac{\partial V}{\partial x}
P+\frac{\partial V}{\partial y} Q-s \left(\frac{\partial P}{\partial
x}+\frac{\partial Q}{\partial y}\right)V\end{equation} does not
change sign and vanishes only on a null measure set. Define
\[L_X(V):=N+L(V),\]
where $N$ is the number of periodic orbits of $X$ contained  in the
set $\mathcal{V}=\{V(x,y)=0\}.$

 Then, the differential system \eqref{eq:bd} has at most $L_X(V)$ periodic
orbits, which are limit cycles. Moreover, each limit cycle not
contained in $\mathcal{V}$ is hyperbolic,  it is contained in one of
the connected components $\mathcal{U}_i$ of $\R^2\setminus
\mathcal{V}$ and, for each $i=1,2,\ldots,j,$ there are at most
$\ell(\mathcal{U}_i)$ limit cycles in the component $\mathcal{U}_i.$
 The stability of each of these limit cycles is given
by the sign of $-VM_s$ on the region $\mathcal{U}_i$.
\end{theorem}

\begin{remark} The function $M_s,$ when $s\le0,$ can also be used to
control te number of limit cycles of \eqref{eq:bd}, see
\cite{Che1997,GasGia2013}. In particular, notice that $M_0=\dot V$
and it can be readily seen that, when $s<0,$ the theorem also works,
giving that $L_X(V)=N.$ In this work, we do not use this range of
values of $s.$ In fact, in most of our applications we
will use $s=1,$ although the values $s=2$ and $s=1/3,$ also will
appear.
\end{remark}

Observe also that, somehow, this version of the Bendixson--Dulac
theorem relates the second part of the Hilbert's 16th problem, which
deals with the number of limit cycles (\cite{Il2002}), with the
first part, that deals with the number and distribution of the ovals
of a planar algebraic curve~(\cite{Wil1978}).

Similarly of what happens when one tries to use Lyapunov functions,
the main difficulty in the above theorem for its practical use is
the choice of the function $V$ and of the positive real number $s.$
In other words, the choice of a suitable Dulac function. As we will
see, the function that gives the curvature of the orbits of
\eqref{eq:bd} is sometimes a good candidate for $V.$

Moreover, the most difficult condition to be checked is that $M_s$
does not change sign. Hence, several approaches try to arrive to
functions for which this question can be more easily studied. For
instance, one of these situations is when it is a function of only one variable or
 the product of two functions of one variable, see again \cite{GasGia2013} for some
examples. Another one is when, from some point of view, we can look
to $M_s$ as a quadratic polynomial.

Finally, notice that, given $V$ and $X,$ the computation of the
number $N$ in $L_X(V)$ is usually not difficult, while $L(V)$
depends on the topology of the set $\{V(x,y)=0\},$ see
Section \ref{ss:LV}. When $V$ is polynomial in one of its
variables, to get an upper bound of $L(V)$ is an affordable task.

\section{Preliminary results}
\label{se:pre}

For the sake of notation, from now on, in this paper we will denote
the partial derivatives as subscripts. Hence, for instance, for
$F=F(x,y),$ $F_x=\frac{\partial F}{\partial x},$ or
$F_{x,y}=\frac{\partial^2 F}{\partial x\partial y}$.

\subsection{Proof of Theorem \ref{th:bd}}

We recall a version of the Bendixson--Dulac theorem
for multiply connected regions, see
\cite{Che1997,GasGia2002,Llo1979}.

\begin{theorem}\label{th:bdmc} Consider a $\mathcal{C}^1$ planar differential system
\begin{equation}\label{eq:bdd}
\dot x= P(x,y),\quad \dot y=Q(x,y),
\end{equation}
defined on $\mathcal{U}\subset\R^2,$ an open subset such that
$\R^2\setminus \mathcal{U}$ has $\ell$ bounded components, and
denote by $X=(P,Q)$ its associated vector field. Let
$B:\mathcal{U}\to\R^+\cup\{0\}$ be a $\mathcal{C}^1$ function such
that
\begin{equation*}
\operatorname{div}(BX)=(BP)_x+(BQ)_y
\end{equation*} does not
change sign and vanishes only on a null measure set. Then the system
\eqref{eq:bdd} has at most $\ell$ limit cycles in $\mathcal{U}.$
Moreover, all of them are hyperbolic and their stability is given by
the sign of $\operatorname{div}(BX).$
\end{theorem}

To prove Theorem \ref{th:bd} first one has to show that the periodic
orbits of \eqref{eq:bd} are either contained in
$\mathcal{V}$ or do not cut this set. This fact follows because
$M_s\big|_{\mathcal{V}}=\nabla V\cdot X=\dot V$ does not change
sign. Hence system \eqref{eq:bd} can have some periodic
orbits contained in $\mathcal{V},$ say that it has~$N,$  and all the
others that are strictly contained in each of the connected
components of $\R^2\setminus \mathcal{V}.$ Fix one of these
connected components, say~$\mathcal{U}_i.$ To control the number of
periodic orbits in this set we will apply Theorem~\ref{th:bdmc} with
$B=|V|^{-1/s}$ and $\mathcal{U}=\mathcal{U}_i.$ Some computations
give that
\begin{equation}\label{eq:div}
\operatorname{div}\Big(|V|^{-1/s}X\Big)=-\frac1s\operatorname{sign}(V)|V|^{-1/s-1}M_s
\end{equation}
and, by hypothesis, this function does not change sign on
$\mathcal{U}_i.$ As a consequence, the maximum number of periodic
orbits in $\mathcal{U}_i$ is $\ell(\mathcal{U}_i),$ as we wanted to
prove. Moreover, by using \eqref{eq:div} and again Theorem
\ref{th:bdmc}  we get that all of them are hyperbolic and their
stability is given by the sign of $-VM_s.$

\begin{remark} Notice that in Theorem \ref{th:bd}
nothing is said about the hyperbolicity of the limit cycles
contained in $\mathcal{V}.$ As we will see in
Corollary~\ref{cor:sab}, they can be hyperbolic or not.
\end{remark}

\subsection{About the practical calculation of $L(V)$}\label{ss:LV} Given an
admissible function $V,$ the computation of $L(V)$ relies on the
study of the topology of each of the connected components
$\mathcal{U}_i,$ of $\R^2\setminus \mathcal{V},$ where $\mathcal{V}=\{V(x,y)=0\}.$
Then  $L(V)$ is the sum of all the
quantities $\ell(\mathcal{U}_i)$, where these values are the number of
bounded components of $\R^2\setminus\mathcal{U}_i.$ In fact, it also
holds that the fundamental group of $\mathcal{U}_i$ is
$\pi_1(\mathcal{U}_i)=\Z\stackrel{\ell)}{*\cdots*}\Z,$ where
$\ell=\ell(\mathcal{U}_i).$
In all concrete situations appearing in this work there is a more
direct way for obtaining $L(V).$ This number is simply the number of
bounded connected components of $\mathcal{V}.$

\subsection{Curvature of the orbits}

It is know that the function
\[
K^{\perp}:=Q^2P_x+P^2Q_y-PQ(P_y+Q_x),
\]
that is the numerator of the curvature of the orbits of the vector
field  $X^\perp=(-Q,P),$ orthogonal to the vector field $X=(P,Q),$
associated to the system~\eqref{eq:bd}, can be used to know the
stability of the periodic orbits of \eqref{eq:bd} and other
dynamical features of its phase portrait, see  \cite{Chi1992,
Dil1950,GasGasGui1996} or \cite[p. 29]{Ye1986}. For instance,
Diliberto in 1950 proved that a limit cycle is hyperbolic and stable
(resp. unstable) if and only if
\[
\int_0^l K^\perp(\gamma(s))\,{\rm d} s<0\quad(\mbox{resp.}\,\, >0),
\]
where $\gamma(s)$ is its parameterization by the arc length and $l$
is its length.

 In this work we will see that the function
\begin{equation}\label{eq:k}
K:=Q^2P_y-P^2Q_x+PQ(P_x-Q_y),
\end{equation}
 proportional to the numerator of the curvature of the orbits of
$X$ is, in several cases, a good candidate for a suitable choice of
$V$ in Theorem~\ref{th:bd}. Notice that $K=Q\dot P -P\dot Q
=Q(P_xP+P_yQ)-P(Q_xP+Q_yQ).$
 As far as we know, this is the first time that this
function $K$ is used to control the number of limit cycles of planar
differential systems. We prove:

\begin{theorem}\label{le:cur} Consider planar system \eqref{eq:bd} of class
$\mathcal{C}^2.$ Assume that the function
\[
D:=P^2Q\big(P_{xx}-2Q_{xy}\big)+PQ^2\big(2P_{xy}-Q_{yy}\big)+Q^3P_{yy}-P^3Q_{xx}
\]
does not change sign and vanishes on a null measure set. Then the
system \eqref{eq:bd} has at most $L_X(V)$ limit cycles, where $V=K$
is given in~\eqref{eq:k} and $L_X(V)$ is defined in Theorem
\ref{th:bd}.\end{theorem}

\begin{proof} By taking
$V=K,$ as in equation \eqref{eq:k}, and $s=1,$ the function $M_{1}$
given in Theorem \ref{th:bd} is $M_{1}=D$ and the theorem follows.
\end{proof}
We will apply this result at the end of Section $3$ for Li\'{e}nard systems and in Section $4$ to rigid systems.

\section{Li\'{e}nard type systems}\label{se:lie}

We present several applications of the Bendixson--Dulac theorem to
two families related with Li\'{e}nard systems.

\subsection{Li\'{e}nard systems with an explicit solution}
We study a family of Li\'{e}nard type equations introduced recently in
\cite{GasSab2019} that includes the Wilson family of Li\'{e}nard
equations (\cite{Wil1964}), which gave the first example of
such equations having an algebraic limit cycle. More concretely, we
consider systems
\begin{equation}\label{eq:wil}
\begin{cases}
\dot x=y-(x^2-1)B(x),\\
\dot y=-x(1+yB(x)),
\end{cases}
\end{equation}
where $B$ is a $\mathcal{C}^1$ function. They have the invariant
algebraic curve $C(x,y)=x^2+y^2-1=0,$ because $C_x P+C_yQ=-2xBC,$
where $X=(P,Q)$ denotes the vector field associated to
\eqref{eq:wil}. Hence, when this system  has not equilibrium points
on the curve, it is a periodic orbit. Moreover, depending on the choice of the function $B$, it can be
a limit cycle.

The following result allows to extend, and to prove in an easier way, the recent
results about the maximum number of limit cycles of the above system
when $B(x)=x^3-bx$ given in \cite{CheChe2020,GasSab2019}.

\begin{theorem}\label{th:wil} Consider the system \eqref{eq:wil}
with $B(x)=x\int_0^x W(t)/t\,{\rm d}t-bx,$ where $W$ is any function
that does not change sign,  vanishes at isolated points, and such
that $B$ is of class $\mathcal{C}^1.$  Then this system has at most
$L+N$ limit cycles, where $L$ is the number of bounded connected
components of the set $\mathcal{B}=\{x\in\R\,:\,
(B(x)+2x)(B(x)-2x)\ge0\}$ plus one, and $N\in\{0,1\}.$ In fact $N=1$
when $\mathcal{C}=\{x^2+y^2-1=0\}$ is free of equilibrium points of
the system, and then this set is one of the limit cycles, and $N=0$
otherwise. Moreover, all the limit cycles but~$\mathcal{C}$ are
hyperbolic and their stability is given by the sign of $VW$ in the
connected component of $\R^2\setminus\{V(x,y)=0\}$ where they lie,
with
\begin{equation}\label{eq:vv}
V=(1-x^2-y^2)\left(x^2+y^2+B(x)y\right).
\end{equation}
\end{theorem}
\begin{proof} Consider the function $V$ given in \eqref{eq:vv} and
$s=1$ in Theorem \ref{th:bd}. Then,
\[
M_1(x,y)=x(x^2+y^2-1)^2\big(B(x)-xB'(x)\big)=-x^2(x^2+y^2-1)^2W(x).
\]
Hence, thanks to the imposed conditions on $W,$ we can apply
Theorem~\ref{th:bd}.  Moreover, since $\mathcal C$ is invariant, and
contained in $\mathcal{V}=\{V(x,y)=0\},$ we have that $N\in\{0,1\}$
and the number of limit cycles of the system  is bounded by
$L(V)+N.$ To get $L(V)$ we study the bounded connected
components of $\mathcal{V},$ see Section \ref{ss:LV}. Notice that
these components are formed by the oval $\mathcal C$ together with
the bounded connected components of $x^2+y^2+B(x)y=0.$ Since this
curve also writes as
\[
y=\frac{-B(x)\pm\sqrt{(B(x)+2x)(B(x)-2x)}}2,
\]
it is clear that these components are obtained by joining the curves
plus and minus defined for $x$ on each of the bounded connected
components of~$\mathcal{B}.$ Hence $L(V)$ is at most $L$ and the
theorem follows.
\end{proof}

The following corollary gives an easier and different proof of all the
results about the maximum number of limit cycles of \eqref{eq:wil}
when $B(x)=x^3-bx,$
\begin{equation}\label{eq:wilc}
\begin{cases}
\dot x=y-(x^2-1)(x^3-bx),\\
\dot y=-x\big(1+y(x^3-bx)\big),
\end{cases}
\end{equation}
obtained in \cite{CheChe2020,GasSab2019}. It also solves in the best
possible way the some times called {\it Coppel's problem} for
polynomial systems,  which in his own words (when restricted to
quadratic systems) says:``Ideally one might hope to characterize
the phase portraits of quadratic systems by means of algebraic
inequalities on the coefficients,'' see \cite{Cop1966}. The relevant
values describing the bifurcations of the limit cycles of this system are
$\underline{b}$ and $b^*,$ see again \cite{GasSab2019}.
 We will show below that the value $\underline b\approx -1.44$ is algebraic. It is the
negative root of the polynomial \begin{equation}\label{eq:p6} 4b^6 -
12b^5 - 4b^4 + 28b^3 + 56b^2 - 72b - 229=0,\end{equation} which is
invariant under the change of variables $b\rightarrow 1-b.$ The
quantity $b(1-b)$ satisfies a third degree polynomial equation and
then it is possible to express all the roots in terms of radicals
but we prefer to omit their explicit expressions
because they are rather complicated. The value $b^*\approx 0.747$ is
the only  zero of the function
$Z:(\underline{b},1-\underline{b})\to\R,$
\[
Z(b)=\int_0^1 \frac
{8(b-3x^2)\sqrt{1-x^2}}{x^8-(2b+1)x^6+(b+2)bx^4-b^2x^2+1}\,{\rm d}
x.
\]
The sign of this function gives the stability of
$\mathcal{C}=\{x^2+y^2-1=0\}$ when $\mathcal{C}$ is a limit cycle. Most probably $b^*$ is a
non-algebraic number. The function $Z$ was obtained
in \cite{GasSab2019} from the integral of the divergence of the
system on the algebraic limit cycle after some algebraic
manipulations. In fact, the discriminant with respect
to $x$ of the denominator of the integrand gives the polynomial of
the left-hand side of~\eqref{eq:p6} that determines
$\underline{b}.$

\begin{cor}\label{cor:sab} System \eqref{eq:wilc} has at most two limit cycles, taking into account their multiplicities.
More concretely:
\begin{enumerate}[(i)]
\item It has  no limit cycle for $b\le \underline{b}.$

\item $\mathcal{C}$ is its only limit cycle, and it is hyperbolic
and attractor when $b\in(\underline{b},0].$

\item It has two limit cycles, one hyperbolic, repeller and surrounded
by $\mathcal{C},$ and $\mathcal{C}$ itself, which is hyperbolic
and attractor, when $b\in(0,b^*).$

\item $\mathcal C$ is its only limit cycle, and it is double and
semi-stable when $b=b^*.$

\item\label{cas5} It has two limit cycles, one hyperbolic, attractor and
surrounding $\mathcal{C},$ and $\mathcal{C}$ itself, which is
hyperbolic and repeller, when $b\in(b^*,1-\underline b).$

\item It has one limit cycle surrounding $\mathcal{C}$ that is
hyperbolic and attractor, when $b\ge 1-\underline{b}.$
\end{enumerate}

\end{cor}

\begin{proof} First, we will prove the most difficult part:
the maximum number of limit cycles of the system is three. This will
essentially be a direct consequence of Theorem \ref{th:wil}. All the other results
about this system can be obtained from the standard techniques of the qualitative theory of planar differential systems.

When $b\le0$ the only limit cycle is $\mathcal{C}$ because in polar
coordinates, $\dot r=r(r^2-1)(b-r^2\cos^2\theta)\cos^2\theta$ does
not vanish outside $\mathcal{C}.$

When $b\ge3/2,$ the maximum number of limit cycles is two. To prove
this we apply Theorem  \ref{th:bd} with $V(x,y)=x^2+y^2-1$ and
$s=1/3.$ Then
\[
M_{1/3}(x,y)= \frac13(x^2+y^2-1)\big((2b-3)x^2+b\big)\ge0.
\]
Since for these values of $b,$ $M_{1/3}$ does not vanish outside
$\mathcal{C}$ the maximum number of limit cycles is two, one being
$\mathcal{C}$ and at most another one can exist, and in this case it
must surround $\mathcal{C}.$

Finally, consider the values of $b\in (0,3/2).$ In fact, we can
consider $b\in[0,2].$ We apply Theorem \ref{th:wil} with
$W(x)=2x^2\ge0.$ We get that $B(x)=x^3-bx,$ and
\begin{align*}
\mathcal{B}&=\{x\in\R\,:\, x^2(x^2-2-b)(x^2+2-b)\ge0\}\\&=
\big(-\infty,-\sqrt{2+b\,}\,\big]\cup\{0\}\cup\big[\sqrt{2+b\,},\infty\big).\end{align*}
Hence, the number of bounded connected components of
$\mathcal{B}$ is one and, as a consequence, $L(V)=2$ and the
maximum number of limit cycles is~three. Also, from the proof we
know that if the three limit cycles  exist, one
is~$\mathcal C,$ there is at most another one, say~$\gamma,$
surrounded by~$\mathcal{C},$ and a third one~$\Gamma,$
surrounding~$\mathcal C$.

To reduce this upper bound of three limit cycles by one it
suffices to consider the stability of the origin, the infinity, the
possible limit cycles and the invariant set $\mathcal{C}.$ In fact
we have that,
\begin{enumerate}[(A)]

\item The stability of the origin is given by the sign of $b.$ Moreover,
it is not difficult to see that when
$b\in(\underline{b},1-\underline{b})$ the origin is the only
equilibrium point of the system and that, otherwise, there are also
other equilibrium points, but all of them are on $\mathcal{C}.$

\item The set $\mathcal C,$ which is always invariant by the flow,
is  a limit cycle if and only if $b\in(\underline b,\overline b).$
Moreover it is hyperbolic and stable if $b\in(\underline{b},b^*),$
hyperbolic and repeller if $b\in(b^*,\overline{b}),$ and semi-stable
and double when $b=b^*.$ In fact, in this later case it is repeller
from its interior and attractor from its exterior, see
\cite{GasSab2019}. Moreover, it is also proved in that paper, that
when $b\ge \overline b,$ the set $\mathcal C,$ that it is no more a
periodic orbit, is also a repeller.

\item The infinity is repeller for $b>0,$ see again \cite{GasSab2019}.

\item For $b\in(0,2),$ whenever they exist, $\gamma$ is  hyperbolic and repeller
 and~$\Gamma$ is hyperbolic and attractor. This is a consequence of
Theorem~\ref{th:wil}, because their respective stabilities are
controlled by the sign of~$VW,$ that coincides with the sign of
$1-x^2-y^2,$ because
\[
VW=2x^2(1-x^2-y^2)(x^2+y^2+(x^3-bx)y),
\]
and for these values of $b$ the limit cycles must lie
in the region $\{x^2+y^2+(x^3-bx)y>0\}$ because it is the only
connected component~$\mathcal{U}$ of $\R^2\setminus\mathcal{V}$ with
$\ell(\mathcal{U})\ne0.$
\item For $b\ge 3/2,$ $\gamma$ never exits and $\Gamma$ is also
hyperbolic and atractor, because as we have proved above by using
Theorem \ref{th:bd}, its stability is also given by the sign of
\[
\qquad -V(x,y)M_{1/3}(x,y)=
-\frac13(x^2+y^2-1)^2\big((2b-3)x^2+b\big)\le0.
\]
\end{enumerate}
For instance we will  prove item $(v).$ All the other cases follow
similarly. First notice that by (B), $\mathcal C$ is
a hyperbolic and repeller limit cycle.  Recall that we already have proved that the system has at most one limit
cycle surrounded by $\mathcal {C},$ and another one surrounding
$\mathcal{C}.$ Moreover, whenever they exist they are hyperbolic and their stabilities are given in (D). By (A) and (C), since the origin is attractor and the infinity is repeller, we get
that there is no limit cycle surrounded by $\mathcal{C}$ and there
is exactly one limit cycle, hyperbolic and stable, surrounding
$\mathcal{C}.$

 \end{proof}

\begin{remark} System \eqref{eq:wilc} can be transformed into the
classical Li\'{e}nard system
\begin{equation}\label{eq:wilc2}
\begin{cases}
\dot x=y-bx+x^3+\frac{4b}3x^3-\frac65x^5,\\
\dot y=-x+b^2x^3-b(2+b)x^5+(1+2b)x^7-x^9.
\end{cases}
\end{equation}
By using Theorem \ref{th:bd} with $s=1$ and $V(x,y)=A(x,y)B(x,y)$
where
\begin{align*}
A(x,y)=&-225+225x^2+25b^2x^6-30bx^8+9x^{10}\\&+(150bx^3-90x^5)y+225y^2,\\
B(x,y)=&
225x^2-75b^2x^4+5b(24+5b)x^6-15(3+2b)x^8+9x^{10}\\&+(-225bx+25(9+6b)x^3-90x^5
)y+225y^3,
\end{align*}
we get that $M_{1}=2x^4A^2(x,y)\ge0.$ Hence, in these variables  an
upper bound of the number of limit cycles of system \eqref{eq:wilc}
can also be obtained. This example  illustrates that although
sometimes  it is difficult to find a suitable $V$ to apply Theorem
\ref{th:bd}, it seems to exist.
\end{remark}

\subsection{Some extended Li\'{e}nard systems}

We consider planar differential equations of the form
\begin{equation}\label{eq:lieg}
\begin{cases}
\dot x=y-|y|^{m}F(x),\\
\dot y=-G'(x)/2,
\end{cases}
\end{equation}
where $F$ and $G'$ are $\mathcal{C}^1$ functions satisfying $F(0)=0$
and $G(x)=x^{2k}+o(x^{2k}),$ $m\in\N\cup\{0\}$ and $k\in\N.$ Notice
that when $G(x)=x^2$ and $m=0,$  they include the classical second
order Li\'{e}nard equations $\ddot x+F'(x)\dot x+ x=0.$ The factor
$|y|^{m}$ is added following the recent work~\cite{VilZan2020},
where this interesting system was studied for the first time. Notice
that if instead of $y-|y|^mF(x)$ we consider the same system but
with the first component equal to $y-y^mF(x),$ then, when $m$ is
odd, it would be invariant by the change of variables and time
$(x,y,t)\to(x,-y,-t)$ and the origin would be a reversible center.

In all our study we skip the case $m=1,$ where the associated vector
field is not of class $\mathcal{C}^1.$ In any case, for $m=1,$ and
on each of the regions $y>0$ and $y<0,$ the vector field is
integrable (it corresponds to a differential equation of separated
variables) and by using the level curves of the corresponding first
integrals, their phase portraits are easier to be studied. This
approach is the one used in \cite{VilZan2020}  for this case, when
$G(x)=x^2.$

\begin{theorem}\label{th:nou} Consider the differential system \eqref{eq:lieg} with $m\ne1.$ If the function
$H:=(m-1)FG'+2F'G$ does not change sign and vanishes at isolated
points, then the system has at most $J$ limit cycles, all of them
hyperbolic, where $J$ is the  number of zeroes of $G'.$ In
particular, if~$G'$ only vanishes at the origin  the differential
system has at most one limit cycle.
\end{theorem}

\begin{proof} We apply the Bendixson--Dulac theorem with
$V(x,y)=G(x)+y^2-y|y|^{m}F(x)$ and $s=1.$ Simple computations give
that
\[
M_{1}=\frac12|y|^{m}H(x).
\]
Therefore,  since $M_{1}$ satisfies the hypothesis of the
Bendixson--Dulac theorem we have already proved that system
\eqref{eq:lieg} has at most $L_X(V)$ limit cycles.  We claim that
$L_X(V)\le J.$  Since the set $\mathcal{V}=\{V(x,y)=0\}$ does not
contain solutions of the differential system the claim will follow
if we prove that $\mathcal{V}$ has at most $J$ bounded
connected components, see Section \ref{ss:LV}. Notice that each of
these components can be an oval, an isolated point, or a more
complicated set.

To prove this last assertion we first count the number of points of
$\mathcal{V}\cap \{x=x_0\},$ taking into account their multiplicity,
and  we call it $K(x_0).$ When $m=0$ it is clear that $K(x_0)\le2,$
because $V(x_0,y)=0$ is a quadratic equation in $y.$ When $m\ge2,$
the equation $V(x_0,y)=0$ splits into two trimonomial equations, one
for $y\ge0$ and another one for $y\le0.$ By applying the Descarte's
rule of signs to both equations, since the monomial $y^2$ appears in
both, it can be seen that $K(x_0)\le3.$

Notice that since on $\mathcal{V},$ $M_1=\dot V,$ each bounded
connected component of $\mathcal{V}$ delimits some region either
positively or negatively invariant, and as a consequence its
interior must contain at least one equilibrium point $(x^*,y^*)$ of
the system. Observe also that even when the system has other
equilibrium points on the line $\{x=x^*\},$ only one connected
component of $\mathcal{V}$ can cut this line, because $K(x^*)\le3.$
Hence the bounded connected components of the set $\mathcal{V}$ must
cut the lines $\{x=x^*\},$ where $G'(x^*)=0,$ and at most one of
them cuts each of the lines. As a consequence, $\mathcal{V}$ has at
most $J$ bounded connected components, and $L(V)\le J$  as we wanted
to prove.
\end{proof}

 Theorem \ref{th:nou} can be applied to  several differential
systems~\eqref{eq:lieg}. For  $m$ and $G$ fixed,  let $W$ be a
function that does not change sign, vanishes at isolated points, and
such that the initial value problem for the  linear differential
equation
\begin{equation}\label{eq:linear}
(m-1)F(x)G'(x)+2F'(x)G(x)=W(x), \quad F(0)=0,
\end{equation}
has a regular solution $F.$ Notice that \eqref{eq:linear} is
singular at the zeroes of $G$ and we impose that $F$ must be smooth
at these points. Then the correspondent differential
system~\eqref{eq:lieg} is under the hypotheses of the theorem. By
using this point of view, we have obtained several families of
differential systems where it is easy to impose that their
corresponding functions $H$ do not change sign and, as a
consequence, Theorem \ref{th:nou} can be applied. We will skip all
the hypotheses that must be added to guarantee the desired property
for $H,$ and the other ones that  the functions $F$ and $G$ must
satisfy to fulfill all the other hypotheses of the theorem, because
the reader can easily figure out them. These families are:

\begin{enumerate}[(i)]

\item When
$F(x)=A^p(x)A'(x)B(x),$  $G(x)=c A^q(x)(A'(x))^2B^2(x),$ and $m=0.$
Then, it holds that
\[
H(x)=(2p-q)c A^{p+q-1}(x)(A'(x))^4B^3(x).
\]

\item
When $m=0,$
\begin{align*}
F(x)&=A^{2p}(x)A'(x)B^{q+1}(x),\,\,\,\mbox{and}\\G(x)&=c
A^{4p}(x)(A'(x))^2B^{q}(x),
\end{align*}
we get that
\[
H(x)=(q+2)c A^{6p}(x)B^{2q}(x)(A'(x))^3B'(x)
\]

\item When $m=2k,$  $G(x)=x^{2k},$ and
\[
F(x)= \frac12 x^{k(1-2k)}\int_0^x y^{k(2k+1)} Z(y)\,{\rm d} y,
\]
we obtain that $H(x)=x^{4k}Z(x).$

\item When $m=0,$ $F(x)=a(x^3/3-x)$ and
\[
G(x)=x^2-\Big(\frac{a^2}8+6b\Big)x^4+\Big(\frac{a^2}{48}+b\Big)x^6,
\]
we get that
\[
H(x)=\frac{a(16-3a^2-144b)}{12}x^4.
\]
\end{enumerate}

Now we will study in more detail some particular sub-cases of the
above families and we will refine the upper bound for their number
of limit cycles given in Theorem \ref{th:nou}.

We start with an example contained in the family
given in item~(i). It corresponds  to $p=1,$ $q=0,$ $c=1/4,$
$A(x)=x^3/3-x^2/2$ and $B(x)=-2,$
 and writes as
\begin{equation}\label{eq:1-3}
\begin{cases}
\dot x=y+\dfrac13 x^3(x-1)(2x-3),\\[0.2cm]
\dot y=-x(x-1)(2x-1),
\end{cases}
\end{equation}
with $G(x)=x^2(1-x)^2.$ We will prove that this system has at most one limit
cycle, hyperbolic and stable. The existence of this limit cycle,
that surrounds the three equilibrium  points of \eqref{eq:1-3}, can
be numerically confirmed.

By using Theorem \ref{th:nou} when $m=0$ and with
\[
V(x,y)=x^2(x-1)^2+\frac13x^3(x-1)(2x-3)y+y^2
\]
we get that $H(x)=-4x^4(x-1)^4<0$ and the maximum number of limit
cycles of the corresponding system is three, which is the number of
zeroes of $G'(x)=2x(x-1)(2x-1).$ This upper bound can be reduced to
two studying in more detail the set $\mathcal{V}.$ This set is
formed by two isolated critical points located at $(0,0)$ and
$(1,0)$, and two disjoint curves going from infinity to infinity.
The point $(0,0)$ is a weak focus and the
 point $(1,0)$ is a strong stable focus. The third critical point, located at $(1/2,-1/24)$, is a saddle point.
   By computing the first Lyapunov quantity associated to the weak focus at
   the origin we conclude that this point is repulsive. In fact,
$\R^2\setminus \mathcal{V}$ is formed by three open sets, two are
simply connected and the third one has two holes (the two critical
points located on the $x$ axes). In short $L(V)=2$ and since $\mathcal V$ does not contain
periodic orbits, the upper bound of two limit cycles follows from
Theorem \ref{th:bd}.

Finally, we prove that one is the actual upper bound for the number
of limit cycles of \eqref{eq:1-3}. By Theorem \ref{th:bd} the
stability of the limit cycles is given by the sign of $
-V(x,y)M_1(x)$ that coincides with the sign of $-H(x)>0.$ Hence all
of them are repelling hyperbolic limit cycles. By
using   the Poincar\'{e}--Bendixson theorem it can be seen that the only
situations compatible with these results are that either
\eqref{eq:1-3} has no limit cycle, or that it has exactly one, as we
wanted to show.

The same tools allow to prove that
\begin{equation*}
\begin{cases}
\dot x=y-b x^3(x-a)(2x-3a),\\
\dot y=-x(x-a)(2x-a),
\end{cases}
\end{equation*}
has at most one limit cycle.
 In fact, notice that when $b=0$
it can be easily integrated. It has two centers and a saddle point,
and the separatrices of this saddle point  form two homoclinic
trajectories which, together with the critical point, have an eight
shape. Numerically, the limit cycle seems to bifurcate for
$b\approx0$ from this double loop and whenever it exists, it
surrounds the three equilibrium points of the system, a saddle, a
strong focus and a weak focus, both with different stabilities.

Similar examples to system \eqref{eq:1-3}, with more equilibrium
points surrounded by a limit cycle and for which Theorem
\ref{th:nou} also works are not difficult to be constructed. For
instance, if we take $m=0,$ $F(x)=cx^3(1-x)(2-x)^3$ and
$G(x)=x^2(1-x)^2(2-x)^2,$ with $|c|<2$, we get a system with five critical points,
two saddles and three foci, for which $H(x)=8cx^4(1-x)^4(2-x)^4.$

\smallskip

Also, because it contains the classical van der Pol differential
equation, we particularize in detail a subfamily of the one given
in item~(ii). If we consider $p=q=0,$ $c=1,$ and $A'=C$ in $(ii)$ we
get the following result.

\begin{cor}\label{co:vdp} Consider the  $\mathcal{C}^1$ differential
system
\begin{equation*}
\begin{cases}
\dot x=y-C(x)B(x),\\
\dot y=-C(x)C'(x),
\end{cases}
\end{equation*}
with $C(0)=0$ and $C'(x)\ne0$ for $x\ne0.$ If $C(x)B'(x)$ does not
change sign and vanishes at isolated points, then this system has at most one
limit cycle and when it exists it is hyperbolic.
\end{cor}

 Notice that the van der Pol equation corresponds to
 $C(x)=x$ and $B(x)=\lambda(x^2/3-1).$ Then
 $C(x)B'(x)=2\lambda x^2/2$, which does not change sign.

\begin{proof}[Proof of Corollary \ref{co:vdp}]
For these particular cases of differential systems contained in the family (ii) we get that
\[
H(x)=2C^{3}(x)B'(x).
\]
Hence, it does not change sign and vanishes at isolated points.
Notice that~$C$ only vanishes at $x=0,$ because if $C(z)=0,$ by
Rolle's theorem $C'$ would vanish at a point between $0$ and $z.$
Hence, $G'(x)=2C(x)C'(x)=0$ only at $x=0,$ and the corollary
follows. Observe also that in this case
$V(x,y)=C^2(x)+y^2-yC(x)B(x),$ and the set $\{V(x,y)=0\}$ has only one bounded connected component, the
origin, and then $L_X(V)=1.$
\end{proof}

\begin{remark}\label{re:vdp} For completeness we reproduce a second easy proof
of the uniqueness and hyperbolicity of the limit cycle of the van der Pol
equation attributed to Cherkas in \cite[p. 105]{Chi2006}. Write the
equation as the system
\begin{equation*}
\begin{cases}
\dot x=y,\\
\dot y=-x-\lambda(x^2-1)y.
\end{cases}
\end{equation*}
By applying the Bendixson--Dulac theorem with $V=x^2+y^2-1$ and
$s=2$ we get that $M_{2}=2\lambda(x^2-1)^2.$   Clearly, the unit
circle is not a periodic orbit of the system, and  $\{V(x,y)=0\}$
has two connected components, one bounded  and simply connected and
a second one, say~$\mathcal{U},$ with
$\ell(\mathcal{U})=1.$ Hence $L_X(V)=1$ and the result follows.
\end{remark}

When $b=0,$ the system introduced in item (iv)
corresponds to the Wilson Li\'{e}nard equation (\cite{Wil1964})  and
when $|a|<2$ it has the algebraic limit cycle
\[
y^2-\frac a6 x^3y+\frac1{144}(a^2x^6+144x^2-576)=0.
\]
Since this limit cycle is also hyperbolic we get that for $|b|$
small enough the limit cycle persists and our theorem applies to get
an upper bound of the total number of limit cycles of the system
when $b\neq0$. We skip more details  because the study of this
system is quite similar to the one that we did for system
\eqref{eq:1-3}.

We end this section studying in more detail the particular family of
differential systems of the form \eqref{eq:lieg},  introduced in
\cite{VilZan2020},
\begin{equation}\label{eq:vil}
\begin{cases}
\dot x=y-\lambda |y|^{m}(x^3-x),\\
\dot y=-x,
\end{cases}
\end{equation}
where $m\ge2$ is an integer number and $\lambda\in\R.$  Notice that
the factor $x^3-x$ in \eqref{eq:vil} can be changed by $c^2x^3-x,$
with another value of $\lambda,$ obtaining the same phase portrait.
This is so, because by doing the change of variables
$(x,y)\to(cx,cy),$ with $c>0,$ the first equation writes as $\dot
x=y-c^{m}\lambda|y|^{m}(c^2x^3-x)$ and the second one remains
invariant.  We do not take the factor as $x^3/3-x,$ which
corresponds to the van der the Pol equation when $m=0,$ simply to keep
the notation of \cite{VilZan2020}. We prove:

\begin{theorem}\label{th:vil} Consider the differential system \eqref{eq:vil} with $m\in\N,$ and $m\ge2.$ Then, the following holds:
\begin{enumerate}[(i)]

\item For  $|\lambda|\ne0$ small enough it has at least one limit cycle.

\item For $|\lambda|\ge \dfrac3{\sqrt 2}\Big(\dfrac 3 m\Big)^{m/2}$ it
has no limit cycle.

\end{enumerate}
\end{theorem}

\begin{proof}  Notice that
the case $\lambda=0$ corresponds to a linear center, and the phase
portrait when $\lambda<0$ can be easily obtained from the one with
$\lambda>0,$ simply by doing the change of variables and time
$(x,y,t)\to(x,-y,-t).$ Then, it suffices to make the proof for the
case $\lambda>0$.

\smallskip

 $(i)$ Given any $\mathcal{C}^1$ perturbed
Hamiltonian systems,
\begin{equation}\label{eq:gen}
 \begin{cases}
     \dot x=\phantom{-}\dfrac{\partial H(x,y)}{\partial y}+\varepsilon R(x,y),
     \\[10pt] \dot y=-\dfrac{\partial H(x,y)}{\partial x}
     +\varepsilon S(x,y),
 \end{cases}
\end{equation}
where $\varepsilon$ is a small parameter, its associated
Melnikov--Poincar\'{e}--Pon\-trya\-gin function is
\[
M(h)=\int _{\gamma(h)} S(x,y)\,{\rm d}x-R(x,y)\,{\rm d}y,
\]
where  the curves $\gamma(h),$ for $h\in (h_0,h_1),$ form a
continuum of ovals contained in $\{H(x,y)=h\}.$  It is known that
each simple zero $ \bar h\in (h_0,h_1)$ of $M$ gives rise to a limit
cycle of \eqref{eq:gen} that tends, when $\epsilon\to0,$ to
$\gamma(\bar h),$ see for instance \cite{ChiLi2007,DumLliArt2006}.

Consider the differential system \eqref{eq:vil} with
$\lambda=\varepsilon.$ By applying the above result with
$H(x,y)=x^2+y^2=h=r^2,$ with $r\in(0,\infty),$ and taking the
parameterization of the level sets as $x=r\cos\theta,$
$y=r\sin\theta,$ we get that
\begin{align*}
M(r^2)&=\int _{x^2+y^2=r^2} |y|^{m}(x^3-x)\,{\rm
d}y\\&=\int_0^{2\pi}\!\!\!
r^{m}|\sin\theta|^m(r^4\cos^4\theta-r^2\cos^2\theta)\,{\rm
d}\theta\\
&=\frac{\sqrt{\pi}}{2}\frac{\Gamma\big({(m+1)}/{2}\big)}{\Gamma\big({(m+6)}/{2}\big)}r^{m+2}\big(3r^2-(m+4)\big),
\end{align*}
where $\Gamma$ is the Euler Gamma function. Hence, for each $m,$
this function has a simple positive zero $r=\sqrt{{(m+4)}/{3}}$,
that gives rise to the desired limit cycle.

\smallskip

$(ii)$ We will apply Theorem \ref{th:bd} with $s=1/3$ and
\[
V(x,y)=\exp\left(\frac{\lambda^2y^{2m}}{9m}\right)\left(3+\lambda x
y|y|^{m-2}\right).
\]
Some calculations give that
\[
M_{1/3}=-\frac19 \exp\left(\frac{\lambda^2y^{2m}}{9m}\right){\lambda
x^2|y|^{m-2}}\left( 2\lambda^2y^{2m}-27y^2+9(m-1) \right).
\]
We need that $M_{1/3}$ does not change sign. Hence, writing $y^2=z$
we want that
\begin{equation}\label{eq:con0}
z^m-\frac{27}{2\lambda^2}z+\frac{9(m-1)}{2\lambda^2}\ge0\quad\mbox{for}\quad
z\ge0.
\end{equation}

Let us prove, that given a real polynomial $P(z)=z^m+bz+c,$ with
$m\ge2,$ it holds that $P(z)\ge0$ for all $z\ge0$ if and only if
\begin{equation}\label{eq:con}
b\ge -m\Big(\frac c{m-1}\Big)^{(m-1)/m}.
\end{equation}
Since $P(0)=c,$ an obvious first condition is that $c\ge0.$ When
$b\ge0$ the result is trivial. When $b<0,$ since $p'(z)=mz^{m-1}+b,$
the function $P$ has a minimum at $z=z_0= (-b/m)^{1/(m-1)}.$ By
imposing that $P(z_0)\ge0,$ \eqref{eq:con} follows after some
straightforward computations.

Condition \eqref{eq:con} applied to the polynomial \eqref{eq:con0}
gives that
\[
-\frac{27}{2\lambda^2}\ge -m\left(\frac
9{2\lambda^2}\right)^{(m-1)/m}.
\]
After some manipulations we get that this inequality is equivalent
to the one given in the statement.

Hence, we are under the hypotheses of Theorem \ref{th:bd}. Moreover,
since $\{V(x,y)=0\}$ does not contain ovals, and all the connected
components of $\R^2\setminus\{V(x,y)=0\}$ are simply connected, we
have that $L_X(V)=0$ and the system has no limit cycle, as we wanted
to prove.
\end{proof}

Similar computations that the ones of item (i) are done in the
Appendix of \cite{VilZan2020} for the case $m=2$.

\begin{remark} The result of item (ii) of Theorem \ref{th:vil} shows
that for any $m\ge2$ there exits a value $\lambda=\lambda^*(m)$ such
that for $|\lambda|\ge \lambda^*(m)$ system \eqref{eq:vil} has no
limit cycle. Moreover, it gives an upper bound of this value.

For $m=2$ it is not a sharp bound, because in \cite{VilZan2020} the
authors study numerically the system and they find that
$\lambda^*(2)\in(1.474,1.475),$ while our bound is $9\sqrt2/4\approx
3.182.$ Nevertheless, by using it, for $m$ big enough, we prove that
the limit cycles only exist for $\lambda$ in an extremely thin
interval of length $3\sqrt{2}\Big(\dfrac 3 m\Big)^{m/2}$ that
decreases exponentially when $m$ grows.\\
It is the first time that the authors see a proof of
the existence of this type of exponentially small intervals for the
presence of limit cycles.
\end{remark}

\subsection{About Massera's theorem}
Consider the classical Li\'{e}nard equation
\begin{equation}\label{eq:liec}
\begin{cases}
\dot x=y-F(x),\\
\dot y=-x,
\end{cases}
\end{equation}
with $F$ a class $\mathcal{C}^2$ function satisfying $F(0)=0.$ We
prove, in a very simple way, the following extension of the
classical Massera's theorem (\cite{Mas1954,San1951}), where the
hyperbolicity of the limit cycle is also guaranteed. Other authors
had already proved this hyperbolicity, see for
instance~\cite{GasGiaLli2009}.

\begin{theorem}\label{th:mas} Consider the differential system \eqref{eq:liec}. If the function
$xF''(x)$ does not change sign and vanishes at isolated points,
then it has at most one limit cycle and when it exists it is
hyperbolic.
\end{theorem}

\begin{proof} We will apply Theorem \ref{th:bd} with $V$ given by the function $K,$ defined in \eqref{eq:k}, associated to the curvature of the system, and
$s=1.$ By using the results of Theorem \ref{le:cur} when $P=y-F(x)$
and $Q=-x,$ we obtain that
\[
V=K= x^2+y^2+F^2-2yF+x\big(y-F\big)F'
\]
and $M_{1}=(y-F)^2xF''.$ Hence $M_{1}$ satisfies the hypothesis of
Theorem~\ref{th:bd}.  To end the proof we have to show that
$L_X(V)\le1.$ Since the set $\mathcal{V}=\{V(x,y)=0\}$
does not contain orbits of the system, it suffices to prove that
$\mathcal{V}$ has at most one bounded connected component, see
Section \ref{ss:LV}. Clearly the points of $\mathcal{V}$ lie on the
two curves
\[
y=F(x)-\frac12xF'(x)\pm\frac12\sqrt{x^2\big((F'(x))^2-4\big)}.
\]
Therefore the bounded connected components of $\mathcal{V}$ are
given by $x=0$ and the bounded subsets of $\R,$ where
$(F'(x))^2-4\ge0.$ These components are either positively or
negatively invariante by the flow of the system because
$M_1\big|_{\mathcal{V}}=\dot V$ does not change sign. Hence they
must surround some of the equilibrium points of the system. Since the origin is
the only equilibrium point, there is at most one of these
components. Hence, $L_X(V)\le1,$ as we wanted to prove.
\end{proof}

We want to emphasize the surprising simplicity of the proof of this classical theorem with the methods employed in this work.

\section{Rigid systems}\label{se:rig}

These systems write as
\begin{equation}\label{eq:rig}
\begin{cases}
\dot x=-y+xF(x,y),\\
\dot y=\phantom{-}x+yF(x,y),
\end{cases}
\end{equation}
where $F$ is an arbitrary smooth function. This name is due to the
fact that in the usual polar coordinates $(r,\theta)$ it holds that
$\dot \theta=1$ and, therefore, their flow rotates around the origin
with constant angular velocity, as a rigid rotation.
Despite their simplicity and the fact that they have the origin as
the unique equilibrium point, the control of the number of limit
cycles of these systems is far to be completely known. They were
introduced by Conti in \cite{Con1994} and studied by several
authors. We prove the following result for them:

\begin{theorem}\label{te:rig} Let $X$ be the vector field associated to
\eqref{eq:rig}. If $F$ is of class $\mathcal{C}^2$ and it holds that
\begin{equation}\label{eq:h}
H:=F_{xx}F_{yy}-F^2_{xy}\ge 0,
\end{equation}
and $H$ vanishes on a null measure set, then \eqref{eq:rig} has at
most $L_X(V)$ limit cycles, where
\begin{equation}\label{eq:vrig}
V=(x^2+y^2)\left(xFF_x+yFF_y+xF_y-yF_x-1-F^2\right)
\end{equation}
and $L_X(V)$ is defined in Theorem  \ref{th:bd}.
\end{theorem}
\begin{proof} We  apply again Theorem \ref{th:bd} with $V=K$ and
$s=1,$ where $K$ is given in \eqref{eq:k}.  We can use the results
of Theorem \ref{le:cur} with $P=-y+xF$ and $Q=x+yF.$ We get that $V$
is as in \eqref{eq:vrig} and \begin{multline*}
M_{1}=D=(x^2+y^2)\Big(\big(x^2F_{xx}+2xyF_{xy}+y^2F_{yy}\big)F^2
\\+2\big((x^2-y^2)F_{xy}+xy(F_{yy}-F_{xx})\big)F\\
+\big(x^2F_{yy}-2xyF_{xy}+y^2F_{xx}\big)\Big).
\end{multline*}
To control the sign of $M_{1}$ we first remove  the factor
$x^2+y^2.$ Notice that the discriminant of the remaining part,
thinking it as a second degree polynomial in $F$, $AF^2+BF+C,$ is
$B^2-4AC=-4(x^2+y^2)^2H\le0.$ Moreover, looking to $A$ and $B$ as
 quadratic homogenous polynomials of the form $ax^2+bxy+cy^2,$ we get that
their corresponding discriminants coincide and are given by
$b^2-4ac=-4H\le0.$ Therefore, the condition \eqref{eq:h} implies
that $M_{1}$ does not change sign and vanishes only on a null
measure set and hence our result follows.
\end{proof}

Notice that the upper bound for the number of limit cycles given in
the above theorem essentially depends on the shape of the set
$\{V(x,y)=0\}.$ To get the actual value of $L_X(V)$ for each case
this set must be carefully studied. We present now a concrete
application when~$F$ is a quadratic polynomial.

\begin{cor}\label{co:rig} Consider the rigid cubic system \eqref{eq:rig}, with
$F=a+bx+cy+dx^2+exy+hy^2.$ If $4dh-e^2>0$ this system has at most one
limit cycle, and when it exists it is hyperbolic.
\end{cor}

This result is not new. It was proved in \cite{GasProTor2005} by
using a totally different approach: the authors transform the system
into a periodic Abel differential equation and then they apply know
results about these equations.  In that work it is also proved that
when $4dh-e^2<0$ there are systems with at least two limit cycles.
Our proof is different and self-contained. Another proof, based on
the study of the stability of the possible periodic orbits, is given
in \cite{FreGasGui2007}.

\begin{proof}  The function $H$ of
 Theorem \ref{te:rig} is $H(x,y)\equiv 4dh-e^2>0$ and hence the system
 has at most $L_X(V)$ limit cycles. Here
  \begin{multline*}
V(x,y)=(x^2+y^2)\Big(-1-a^2+(c-ab)x-(ac+b)y+e(x^2-y^2)+2(h-d)xy\\
+(bx+cy)(dx^2+exy+hy^2)+(dx^2+exy+hy^2)^2\Big).
 \end{multline*}
It is easy to verify that the set $\{V(x,y)=0\}$ does no contain
orbits of the system.

 As the origin is the unique finite critical
point of the system and $M_{1}$ does not change sign, the bounded connected components of the set
$\mathcal{V}=\{V(x,y)=0\}$ must surround the equilibrium point. In
principle,
 from the degree of $V(x,y)$ we can conclude that the maximum number of them is three,
 being the origin one of  these components. But it is easy to show that there are at most
 two bounded connected components. This is so, because if we take $y=0$ in the second factor of $V$ we obtain a
polynomial in $x$ of degree four,
  where the coefficient of $x^4$ and the independent term are of opposite sign. Then it is not possible to have two
   positive roots and two negative roots at the same time. Therefore, the number of connected components in the set $\mathcal{V}$ is at most two,
   one of them being the origin.

In the case where the second bounded connected component exists it
is not difficult to show that   a limit cycle exterior to it cannot
exist. The first step is to determine the stability of infinity.
Writing the system in polar coordinates it is possible to show that,
if $4dh-e^2>0$, the infinity is an attractor for $d>0$ and it is
repulsive for $d<0$. Moreover,  it can be seen by using that
$M_1\big|_{\mathcal{V}}=\dot V,$ that the flow associated to the
system traverses this second bounded component forward for $d>0$ and
inward for $d<0$. As between this bounded connected
component of $\mathcal V$ and infinity only one limit cycle can
exist, and if it exists it is hyperbolic, taking into account the
stability of infinity we conclude that no limit cycle exists in this
region. We conclude then that the system can have at most one limit
cycle. As the origin is an attractor  for $a<0$ and it is repulsive
for $a>0,$ the limit cycle appears via a Hopf bifurcation at the
origin and must be located  in the interior of the non trivial
bounded connected component of the set $\mathcal{V}.$ In conclusion,
the limit cycle exits if and only if $ad<0$ and it is unique.
\end{proof}
We end this work with a second corollary of Theorem \ref{te:rig}
that also covers  some non-polynomial rigid differential systems.

\begin{cor}\label{co:rig2} Consider the differential system
\eqref{eq:rig}, with
$F(x,y)=f(x)+g(y),$ where $f(x)=\sum_{k=0}^{2n} f_kx^k,$ with
$f_k\ge0, k\ge2$ and $f_{2n}>0$.  We assume that $f''(x)\ge0$, $g$ is of class $\mathcal{C}^2,$
with $g''(y)\ge0$ and it vanishes only at isolated points.
Then the system has at most two limit cycles. Moreover, if there exists
$R>0$ such that for all $(x,y)\in\R^2$ with $x^2+y^2\ge R^2$ it
holds that $F(x,y)\ge c>0,$ then the corresponding differential
system has at most one limit cycle, and when it exists, it is
hyperbolic.
\end{cor}

\begin{proof} For this case the function $H$ given in
  Theorem \ref{te:rig} is $H(x,y)=f''(x)g''(y)\ge0$  and it vanishes on
  a null measure set. Hence the system has at most $L_X(V)$ limit cycles, where $V$  is
the function
  given in \eqref{eq:vrig}. To study
  $L_X(V),$ notice first  that it is not restrictive to assume that
$g(0)=0.$  Then,
  $V(x,0)=x^2W(x),$ where
  \begin{equation*}
W(x)=x f(x)f'(x)+g'(0)x-1-f^2(x).
  \end{equation*}
Since when $f(x)=f_kx^k$ it holds that $xf'(x)-f(x)=(k-1)f_kx^k,$ we
get easily that $W(x)=\sum_{j=0}^{4n} w_k x^k,$ where all $w_k\ge0$
for $k\ge2$ and $w_0=-1-f^2(0)<0.$ Hence, by the Descarte's rule of
signs the number of positive roots of $W$ is 1. As a consequence,
the set $\mathcal{V}=\{V(x,y)=0\}$ has at most one bounded connected component surrounding the origin, different from
the origin itself. Recall that $V$ has the factor $x^2+y^2.$ Notice
that the set $\mathcal{V},$ which is not invariant by the flow of
$X,$ can not contain other bounded connected
components. This is so, because $M_{1}$ does not change sign, and
$M_1\big|_{\mathcal{V}}=\dot V.$ Therefore, these  connected
components must surround some equilibrium point of $X,$ but the
origin is the only one. As a consequence of the above reasoning
$L_X(V)\le 2,$ see Section \ref{ss:LV}.

Let us prove now that under the hypothesis on the growth of $F$ the
maximum number of limit cycles is 1. Notice that if
$r=\sqrt{x^2+y^2},$ it holds that $\dot r=
rF(r\cos\theta,r\sin\theta)\ge cr.$ Hence, the infinity is an
attractor, and we can use similar arguments that in the proof of
Corollary~\ref{co:rig} to show that when $L_X(V)=2$ the differential
system has no limit cycle in the unbounded components of
$\R^2\setminus\mathcal{V}.$ Therefore, the existence of at most one
limit cycle follows.
\end{proof}

\subsection*{Acknowledgments}
This work has received funding from the Ministerio de Ciencia e
Innovaci\'{o}n (PID2019-104658GB-I00 grant) and the Ag\`{e}ncia de Gesti\'{o}
d'Ajuts Universitaris i de Recerca (2017 SGR 1617 grant).

\end{document}